\documentclass[10pt]{amsart}
\input xy
\xyoption{all}

\usepackage{amssymb,epsfig,newlfont,amsmath}

\usepackage[english]{babel}

\usepackage[pagewise]{lineno}

\newtheorem{theorem}{Theorem}[section]
\newtheorem{lemma}[theorem]{Lemma}

\newtheorem{proposition}[theorem]{Proposition}

\newtheorem*{theorem*}{Theorem}
\newtheorem*{lemma*}{Lemma}
\newtheorem*{corollary*}{Corollary}
\newtheorem*{proposition*}{Proposition}

\theoremstyle{definition}

\theoremstyle{remark}
\newtheorem{remark}[theorem]{Remark}

\newtheorem*{example1*}{Example -- the case $H=GL(4)$}
\newtheorem*{example2*}{Example -- the case $H'=GL(2,D)$, $D$ a quaternion division algebra}

\numberwithin{equation}{section}

\newcommand {\Q} {\mathbb Q}
\newcommand {\R} {\mathbb R}
\newcommand {\C} {\mathbb C}
\newcommand {\Z} {\mathbb Z}
\newcommand {\A} {\mathbb A}

\newcommand {\G} {\mathbb G}

\newcommand{\sR}{\mathsf{R}}

\newcommand{\si}{\mathsf{i}}
\newcommand{\sj}{\mathsf{j}}

\title[On arithmetically defined hyperbolic $5$-manifolds]{On arithmetically defined hyperbolic $5$-manifolds arising from maximal orders in definite $\Q$-algebras} 
\begin{document}

\author{Joachim Schwermer}
\address{Faculty of Mathematics, University of Vienna, Oskar-Morgenstern-Platz 1, A-1090 Vienna  \and Max-Planck-Institute 
for Mathematics,  Vivatsgasse 7, D-53111 Bonn, Germany.}
\email{Joachim.Schwermer@univie.ac.at}

\thanks{ The author  gratefully acknowledges the support of the Max-Planck-Institute for Mathematics, 
Bonn, in 2024 when this paper was finalised.}


\date{\today}

\begin{abstract}
Using the quaternionic formalism for the description of the group of isometries of hyperbolic $5$-space we consider arithmetically defined $5$-dimensional  hyperbolic
manifolds which are non-compact but of finite volume. They arise from maximal orders $\Lambda$ in the central simple algebra $M_2(D)$ of degree $4$ where $D$ denotes a definite quaternion $\Q$-algebra.  
The affine $\Z$-group scheme $SL_{\Lambda}$ determines an integral structure for  the algebraic $\Q$-group $G = SL_{\Lambda} \times_{\Z} \Q$ obtained by base change. The group $G$ is an inner form of the special linear $\Q$-group $SL_4$. Each torsion-free subgroup $\Gamma \subset SL_{\Lambda}(\Z)$ determines a hyperbolic $5$-manifold, to be denoted $X_G/\Gamma$. Given a principal congruence 
 subgroup $\Gamma(\frak{p}^e)$, we determine the number of ends and the dimensions  of the cohomology groups at infinity of the manifold $X_G/\Gamma(\frak{p}^e)$.
\end{abstract}

\maketitle

\section{introduction}The following fact is a special case of a general result regarding the  nature  of quadratic spaces of dimension $6$ over a field $K$ with trivial Arf invariant and the corresponding spin groups.
Given such a quadratic $K$-space $(V, q)$   which represents 1 and whose Witt index is one, there exists a quaternion division $K$-algebra $B$ with norm form $n_B$ such that $(V, q)$ is similar to $H(K) \perp (B, n_B)$ where $H(K)$ denotes the hyperbolic plane. The spin group  
$\mathrm{Spin}(V, q)$ is isomorphic to the special linear group of degree two over $B$ (see \cite{Tama}, \cite[Chap. V, Sect. 5.6]{Knus}). In the case $K = \Q$, and $B$ a  definite quaternion $\Q$-algebra, this relation discloses an approach to arithmetically defined hyperbolic $5$-manifolds which commences from maximal orders in $B$. 
\footnote{This point of view reminds  one of arithmetic hyperbolic $3$-manifolds via Bianchi groups attached to the ring of integers of an imaginary quadratic number field
 (see, e.g. \cite{MacReid}).}
 These hyperbolic $5$-manifolds are non-compact but of finite volume. In this paper we illustrate in a simple higher dimensional case the utility of an approach via simple algebras to the study of the geometry and cohomology of arithmetic quotients.

We consider   locally symmetric spaces $X_G/\Gamma$ associated with arithmetic subgroups $\Gamma$ of the special linear algebraic $\Q$-group $G = SL_{M_2(D)}$, attached to a quaternion division $\Q$-algebra $D$. The group $G$ is an inner form of $SL_4$, and $G$ is $\Q$-simple,  of $\Q$-rank one, and non-split over $\Q$. 
 If the quaternion division $\Q$-algebra $D$ is  definite, then $X_G$ is hyperbolic $5$-space, otherwise $X_G$ is the $9$-dimensional symmetric space attached to the real Lie group $SL_4(\R)$.
 
 Given a maximal order $\Lambda_D$ in $D$, the corresponding order $\Lambda = M_2(\Lambda_D)$ defines a $\Z$-group scheme $SL_{\Lambda}$. Via base change, we have 
 $G = SL_{\Lambda} \times_{\Z} \Q = SL_{M_2(D)}$, that is, the order $\Lambda$ provides an integral structure for $G$. The group
$\Gamma_{\Lambda}: = SL_{\Lambda}(\Z)$ is an arithmetic subgroup of $G(\Q)$. Any subgroup $\Gamma$ of $\Gamma_{\Lambda}$ of finite index acts properly on the  symmetric space $X_G$, and, if $\Gamma$ is torsion-free, the quotient $X_G/\Gamma$ carries the structure of a complete Riemannian manifold of finite volume but non-compact.
 As a result of reduction theory, there exists  an open subset $Y_{\Gamma} \subset X_G/\Gamma$ such that its closure $\overline{Y}_{\Gamma}$ is a compact manifold with boundary $\partial \overline{Y}_{\Gamma}$, and the inclusion $\overline{Y}_{\Gamma} \longrightarrow X_G/\Gamma$ is a homotopy equivalence. 

We   analyse the structure  of the connected components  $Y^{[P]}$ of the boundary $\partial \overline{Y}_{\Gamma}$, parametrised by  the finite set $\frak{P}/\Gamma$ of 
$\Gamma$-conjugacy classes of minimal parabolic $\Q$-subgroups of $G$. If $D$ is definite, $\partial \overline{Y}_{\Gamma}$ is the disjoint union of $4$-dimensional real tori.
If $D$ is an indefinite $\Q$-algebra,  a connected component $Y^{[P]}$ of  $\partial \overline{Y}_{\Gamma}$ is diffeomorphic to the total space of a fibre bundle whose base space 
 is diffeomorphic to an arithmetically defined compact quotient of the product $\mathrm{H}^2 \times \mathrm{H}^2$ of two upper half planes 
and whose  compact fibres are equal to  $4$-dimensional tori (see Proposition \ref{boundary}).

 Next, given a  principal congruence subgroup $ \Gamma_{\Lambda}(\frak{p}^e)$ of $G(\Q)$, $\frak{p}^e$ a power of a prime ideal $\frak{p}$,
we determine the number $\mathrm{cusp}(G, \Gamma_{\Lambda}(\frak{p}^e)): =  \frak{P}/\Gamma_{\Lambda}(\frak{p}^e)$ of  $\Gamma_{\Lambda}(\frak{p}^e)$-conjugacy classes of minimal parabolic $\Q$-subgroups of $G$, that is, the number of  connected components of the  boundary $\partial \overline{Y}_{\Gamma_{\Lambda}(\frak{p}^e)} $ of the compact manifold $\overline{Y}_{\Gamma_{\Lambda}(\frak{p}^e)}$.

  \begin{proposition}
  Let $D$ be  a   central simple  division $\Q$-algebra of degree two, and let $h_D$ be the class number of $D$. Let  $G = SL_{\Lambda} \times_{\Z} \Q$ be the  algebraic $\Q$-group attached to a maximal order $\Lambda = M_2(\Lambda_D)$ in the central simple $\Q$-algebra $M= M_2(D)$. Then, given a  principal congruence subgroup $ \Gamma_{\Lambda}(\frak{p}^e)$ of $G(\Q)$, $\frak{p}^e$ a power of a prime ideal $\frak{p}$,

\begin{itemize}
\item[-]{if $D$ is a definite $\Q$-algebra, we have $\mathrm{cusp}(G, \Gamma_{\Lambda}(\frak{p}^e))  = h^2_D N(\frak{p^e})^{-4} [\Gamma_{\Lambda} : \Gamma_{\Lambda}(\frak{p^e})]$.}
\item[-]{if $D$ is an indefinite $\Q$-algebra, we have $\mathrm{cusp}(G, \Gamma_{\Lambda}(\frak{p}^e))  = h_D \mu_{\Lambda_D}(\frak{p}^e)^{-1} N(\frak{p}^e)^{-4} [\Gamma_{\Lambda} : \Gamma_{\Lambda}(\frak{p}^e)]$.}
\end{itemize}
where $\mu_{\Lambda_D}(\frak{p}^e)$ denotes the cardinality of the finite group  $\Lambda^{\times}_D/\Lambda^{\times}_D(\frak{p}^e)$, i.e. the units in $\Lambda_D$  modulo the elements which are congruent to one modulo $\frak{p}^e$.
\end{proposition}
Explicit formulas for the index $[\Gamma_{\Lambda} : \Gamma_{\Lambda}(\frak{p})]$ are given in Lemma \ref{index}; these depend on the splitting behaviour of the prime ideal $\frak{p}$ in $D$. In the first case, families of examples of definite $\Q$-algebras are given in Section \ref{examples}, including  their class number $h_D$.
 \vspace{0,3cm}
 
 Given $\Gamma \subset \Gamma_{\Lambda}$ of finite index, the  pair $(\overline{Y}_{\Gamma}, \partial \overline{Y}_{\Gamma})$ gives rise to 
 the  long exact sequence in cohomology 
$$
\longrightarrow H^q(\overline{Y}_{\Gamma}, \partial \overline{Y}_{\Gamma}, \C) \longrightarrow 
H^q(\overline{Y}_{\Gamma}, \C)  \overset{r^q}{\longrightarrow} H^q(\partial \overline{Y}_{\Gamma}, \C) \longrightarrow  H^{q+1}(\overline{Y}_{\Gamma}, \partial \overline{Y}_{\Gamma}, \C) \longrightarrow
$$
where $r{^q}: H^q(\overline{Y}_{\Gamma}, \C)  \longrightarrow H^q(\partial \overline{Y}_{\Gamma}, \C)$ denotes the natural homomorphism obtained by restriction.
In view of the isomorphism  $H^q(\overline{Y}_{\Gamma}, \C) \cong H^q(X_G/\Gamma, \C)$, the dimension of the terms 
$$\sR^q(\Gamma_{\Lambda}(\frak{p}^e)): = \mathrm{im}(r^q: H^q(\overline{Y}_{\Gamma}, \C)  \longrightarrow H^q(\partial \overline{Y}_{\Gamma}, \C))$$
gives information about  the size of the ``cohomology of $\Gamma$ at infinity", that is, the space of classes in  $H^q(X_G/\Gamma, \C)$ which restrict non-trivially to 
$H^q(\partial \overline{Y}_{\Gamma}, \C)$.

 In the case of a definite quaternion $\Q$-algebra we obtain the following result.
 \begin{theorem}Let $D$ be a definite quaternion $\Q$-algebra. Let $G = G_{\Lambda} \times_{\Z} \Q$ be the algebraic $\Q$-group associated with a maximal order $\Lambda$ in the central simple $\Q$-algebra $M_2(D)$.
 Given the  principal congruence subgroup $\Gamma_{\Lambda}(\frak{p}^e) \subset \Gamma_{\Lambda} = G_{\Lambda}(\Z)$, $\frak{p}^e$ a power of a  prime ideal $\frak{p}$, 
 we write 
 $\sR^q(\Gamma_{\Lambda}(\frak{p}^e)): =  \mathrm{im}(r^q: H^q(\overline{Y}_{\Gamma}, \C)  \longrightarrow H^q(\partial \overline{Y}_{\Gamma}, \C))$ for the  image of the map $r^q$. Then we have 
\begin{itemize}
\item[-] {$\dim \sR^0(\Gamma_{\Lambda}(\frak{p}^e)) = 1$, and $\sR^4(\Gamma_{\Lambda}(\frak{p}^e))$ is of codimension one in $H^*(\partial \overline{Y}_{\Gamma_{\Lambda}(\frak{p}^e)}, \C)$}
\item[-] {$\dim \sR^1(\Gamma_{\Lambda}(\frak{p}^e)) + \dim \sR^3(\Gamma_{\Lambda}(\frak{p}^e)) = \dim H^3(\partial \overline{Y}_{\Gamma_{\Lambda}(\frak{p}^e)}) =
4 h_D^2 N(\frak{p}^e)^{-4} [\Gamma_{\Lambda} : \Gamma_{\Lambda}(\frak{p}^e)]$}
\item[-] {$\dim \sR^2(\Gamma_{\Lambda}(\frak{p}^e))  = \frac{1}{2} \dim H^2(\partial \overline{Y}_{\Gamma_{\Lambda}(\frak{p}^e)}) =
3 h_D^2 N(\frak{p}^e)^{-4} [\Gamma_{\Lambda} : \Gamma_{\Lambda}(\frak{p}^e)]$}
\end{itemize}
 \end{theorem}
 In Section \ref{examples} we give some explicit examples, in particular, some with class number $h_D = 1$.
 
  With regard to  an indefinite quaternion $\Q$-algebra, we refrain from working out this topological approach to the cohomology at infinity.   
 In this case, the interpretation of $H^q(X_G/\Gamma, \C)$ in terms of automorphic forms, in particular, the use of Eisenstein series (see, e.g., \cite{GrSch}, \cite{S2}), gives far better insight.
see Remarks \ref{indef} and \ref{Eis}.

\section*{Notation and conventions} 

(1)  Let $k$ be an algebraic number field,  and let $\mathcal{O}_k$ denote its ring of integers.
The set of places of $k$ will be denoted by $V_k$, and  $V_{k, \infty}$ (resp. $V_{k, f}$) refers to the subsets  of archimedean (resp. non-archimedean) places of $k$. Given a place $v \in V_k$, the
  completion of $k$ with respect to the absolute value $\vert \cdot \vert$ corresponding to $v$ is denoted by $k_v$. For a finite place $v \in V_{k, f}$ we write $\mathcal{O}_{k,v}$ for the valuation ring in $k_{v}$. The unique maximal ideal in $\mathcal{O}_v$ is denoted by $\frak{P}_v$. 
   
    Let $\mathbb{A}_k$ (resp.~$\mathbb{I}_k$) be the ring of
ad\`{e}les (resp.~the group of id\`{e}les) of $k$. We denote by $\A_{k, \infty} = \prod_{v \in V_{k, \infty}}k_v$ the archimdean component of the ring $\A_k$, and by 
$\mathbb{A}_{k,f}$ the finite ad\`{e}les of $k$. There is the usual decomposition of $\A_k$  into the archimedean and the non-archimedean part $\A = \A_{k, \infty} \times \A_{k, f}$. 

(2) Let $A$ be a commutative ring with identity. By definition, an algebraic $A$-group $G$ is an affine $A$-group scheme that is of finite type as an affine scheme over $A$. For every commutative associative $A$-algebra $R$ we write $G(R)$ for the group of $R$-valued points of $G$.
If $A = K$ is a field, we additionally suppose the defining condition that $G$ is smooth.

\section{An inner form of the algebraic $\Q$-group $SL_4$}\label{schemes} 
Given the  algebraic number field $k = \Q$ of rational numbers with ring of integers $\mathcal{O}_k = \Z$, let 
 $D$ be a  central simple  $\Q$-algebra of degree two, that is, $\dim_{\Q} D = 4$.
We associate with a given maximal $\Z$-order $\Lambda$ in the central simple $\Q$-algebra $M: = M_2(D)$ an affine $\Z$-group scheme $SL_{\Lambda}$ of finite type. 
One obtains an integral structure on the special linear $\Q$-group $SL_M = SL_{\Lambda} \times_{\Z} \Q$.   The group $G: = SL_M$ is a $\Q$-simple simply connected group of $\Q$-rank one. It is an inner form of the algebraic $\Q$-group $SL_4$.

\subsection{Central simple algebras of degree two}
Given any field $F$ of characteristic $\mathrm{char} F = 0$, let $Q$ be a quaternion $F$-algebra, that is, a central simple $F$-algebra of degree two. 
In terms of generators and relations, the $F$-algebra $Q$, viewed as an $F$-vector space, has a basis $\mathsf{1}, \si, \sj, \si \sj$ subject to the relations $\si^2 = a, \sj^2 = b, \si \sj = - \sj \si$ for some elements $a, b \in F^{\times}$. The entire multiplication rules for $Q$ are determined by the relations as given, linearity and associativity.
Although the $F$-algebra does not uniquely determine the elements $a, b \in F^{\times}$ we may use the notation $Q = Q(a, b \vert F)$, thereby emphasising the choice of $a, b \in F^{\times}$.

A given quaternion $F$-algebra $Q$ is either isomorphic to the $F$-algebra $M_2(F)$ of $(2 \times 2)$-matrices with entries in $F$ or a central division $F$-algebra. 
The $F$-algebra $Q$ is endowed with the reduced norm, a multiplicative map $\mathrm{nrd}_{Q/F}: Q \longrightarrow F$ (see \cite[Sect. 9]{Re}). In fact, the reduced norm is a quadratic form on $Q$. The reduced norm of a given element $x \in Q$ vanishes if and only if $x$ is not invertible. Therefore, if $Q$ is a division algebra the reduced norm of $x \in Q$ vanishes if and only if $x = 0$. The reduced trace, to be denoted 
 $\mathrm{red\; tr}_{Q/F}$,  is a linear form on $Q$.
The reduced characteristic polynomial of a given element $q \in Q$ has the form
\begin{equation}
\mathrm{red} \chi_{q, Q/F} = X^2 - \mathrm{red\; tr}_{Q/F}(q) X + \mathrm{nrd}_{Q/F}(q).
\end{equation}

\subsection{Ramification} Let $k$ be an algebraic number field, and let $\mathcal{O}_k$ denote its ring of integers. Let 
 $D$ be a  central simple  $k$-algebra of degree two.
Given a place $v \in V_{k}$, the local analogue $D_v = D \otimes_k k_v$ is a central simple $k_v$-algebra of degree two, that is, $\dim_{k_v}D_v = 4$. If $v \in V_{k, \infty}$ is a complex place, we have $D_v \cong M_2(\C)$.
 If $v \in V_{k, \infty}$ is a real place, the $\R$-algebra $D_v$ is either isomorphic to $M_2(\R)$ or the division algebra $\mathbb{H} = Q(-1, -1 \vert \R)$ of Hamilton quaternions.  A similar dichotomy exists in the case of a non-archimedean place $v \in V_{k, f}$. 
 For each local field $k_v$, $v \in V_{k, f}$, there is, up to isomorphism,  a unique quaternion division algebra $C_v$ over $k_v$.  Using the unique unramified quadratic extension $k_v(\sqrt{\alpha})/k_v$ where $\alpha$ is a unit in the valuation ring $\mathcal{O}_v$ of $k_v$, $C_v$ can be constructed as a cyclic algebra. Therefore,
 the quaternion algebra $D_v$ is isomorphic either to $M_2(k_v)$ or the unique division $k_v$-algebra  $C_v$. 

We say that $D$ ramifies at a place $v \in V_k$, or that $v$ is ramified in $D$, if $D_v$ is isomorphic to a division algebra, otherwise $D$ splits at
$v \in V_k$. A given central simple $k$-algebra $D$ of degree two splits at all but a finite number of places, and the set $\mathrm{Ram}(D) = \{v \in V_k\; \vert \; D \; \text{ramifies at}\;  v \in V_k \}$
has even cardinality. The isomorphism class of the algebra $D$ over $k$ is determined by the ramification set $\mathrm{Ram}(D)$. Furthermore, given a set of places
$S \subset V_k$ where $S$ has even cardinality, there exists a quaternion $k$-algebra with ramification set equal to $S$.

We call a central simple $k$-algebra $D$ of degree two a totally definite quaternion algebra if $D$ ramifies at each  archimedean place $v \in V_{k, \infty}$, that is, 
$D_v \cong \mathbb{H}$. Consequently, every archimedean place of $k$ is real place. These central simple $k$-algebras play the role of exception in the theory of simple algebras over number fields, in particular, with regard to orders (For details we refer to \cite[Section 34]{Re}).

We say that a central simple $k$-algebra $D$ is indefinite if $D$ is not totally definite.

\subsection{The group $SL_M$ and its parabolic subgroups} 
 Let $D$ be a central simple  $\Q$-algebra  of degree two, and let $M: = M_2(D)$ be the central simple $\Q$-algebra of $(2 \times 2)$-matrices with entries in $D$.
 Let  $GL_M$ be the affine $\Q$-group scheme of finite type associated with the assocoative $\Q$-algebra $M$ (see \cite[Sect. 8.3]{Schbook}). This is a connected algebraic $\Q$-group whose group of  $\Q$-rational points is  the group $GL(2, D)$ of $(2 \times 2)$-matrices with entries in $D$.
We fix a maximal $\Q$-split torus $S$ in $GL_M$ subject to the condition
$
S(\Q) = \left \{ g =  \bigl ( \begin{smallmatrix}
\lambda & 0 \\
0 & \mu \end{smallmatrix} \bigr ) \; \vert \; \lambda, \mu  \in
\Q^{\times}1_{D} \right \}.
$
 The
non-trivial character $\alpha: S/\Q \rightarrow \mathbb{G}_{m}/\Q$, defined
by the assignment $ \bigl (
\begin{smallmatrix}
\lambda & 0 \\
0 & \mu \end{smallmatrix} \bigr ) \mapsto \lambda \mu ^{-1}$, forms a basis for the set of roots of $GL_M$.
We denote by $Q_{0}$ the minimal parabolic $\Q$-subgroup of $GL_M$ which is determined by $\{ \alpha\}$. We have a Levi decomposition of $Q_{0}$ into the semi-direct product $Q_0 = Z_{GL_M}(S) N_0$ of its unipotent radical $N_0$ by
 the centraliser $L_{Q_0} = Z_{GL_M}(S)$ of $S$.
The group of $\Q$-rational points of the centraliser $Z_{GL_M}(S)$ of $S$ is given by
 $Z_{GL_M}(S)(\Q) = \left \{ g =  \bigl ( \begin{smallmatrix}
x & 0 \\
0 & y \end{smallmatrix} \bigr ) \; \vert \; x, y  \in
D^{\times} \right \}.
$

The determinant defines a homomorphism of algebraic $\Q$-groups $\det: GL_M \longrightarrow \G_m$ into the multiplicative group. We call its scheme-theoretic kernel the special linear group of
$M$, to be denoted $SL_M$ (cf. \cite[Examples A 1.8.]{Schbook}).
The group $SL_M$ is a $\Q$-simple simply connected algebraic group of $\Q$-rank one. We fix the maximal $\Q$-split torus $T$ of $SL_M$, whose $\Q$-rational points are equal to $T(\Q) = SL_A(\Q) \cap S(\Q)$, hence, 
$
T(\Q) = \left \{ g =  \bigl ( \begin{smallmatrix}
\lambda & 0 \\
0 & \lambda^{-1} \end{smallmatrix} \bigr ) \; \vert \; \lambda  \in
\Q^{\times}1_{D} \right \}.
$
 The restriction of the root $\alpha$ of $GL_M$ to $T$, denoted by the same letter, is a basis for the set of roots for $SL_M$ with respect to $T$. We denote 
 the minimal parabolic $\Q$-subgroup which corresponds to $\alpha$ by $P_0$. It has a Levi decomposition as a semi-direct product $P_0 = Z_{SL_M}(T) N_0$ of its unipotent radical $N_0$ by
 the centraliser $L_0 = Z_{SL_M}(T)$ of $T$. 
 We denote by $M_0$
  the connected component of the group $(\cap_{\chi \in X_{\Q}(L_0)} \ker \chi)$.
 The $\Q$-rational points of $L_0 = Z_{SL_M}(T)$ resp. $M_0$ are given by
 $$L_0(\Q) = \left \{ g =  \bigl ( \begin{smallmatrix}
x & 0 \\
0 & x^{-1} \end{smallmatrix} \bigr ) \; \vert \; x  \in
D^{\times} \right \} \quad \text{resp.} \quad M_0(\Q) = \left \{ g =  \bigl ( \begin{smallmatrix}
x & 0 \\
0 & x^{-1} \end{smallmatrix} \bigr ) \; \vert \; x  \in
D^{\times}, \mathrm{nrd}_{D/\Q}(x) = 1 \right \}.
$$
The group $M_0$  is the largest connected anisotropic subgroup of $L_0$, the intersection $M_0(\Q) \cap T(\Q)$ is finite, and $L_0 = M_0 T$. Thus, we have $P_0 = M_0 T N_0$.
We call $P_0$ the standard minimal parabolic subgroup of $SL_M$.

\subsection{Orders in a central simple $\Q$-algebra.}\label{orders} Let $A$ be a central simple $\Q$-algebra of degree two.
By definition, a $\Z$-order $\Delta$ in $A$ is a subring of $A$ with 
$1_{\Delta} = 1_A$ and such that $\Delta$ is a complete $\Z$-lattice in $A$. 
For each $x \in \Delta$ the reduced characteristic polynomial $\mathrm{red} \chi_{x, A/k}$ has coefficients in $\Z$. In particular, given the reduced norm
$
\mathrm{nrd}_{A/k}: A \longrightarrow k,
$
we have $\mathrm{nrd}_{A/k}(x) \in \Z$. Every  finite-dimensional central simple $k$-algebra $A$ contains $\Z$-orders, and each of them is contained 
in a maximal $\Z$-order. 

A left $\Delta$-ideal  in $A$ is a left $\Delta$-lattice $N \subset A$ such that $\Q N = A$, or, equivalently, since $\Delta$ is finitely generated over $\Z$, 
a complete $\Z$-lattice in $A$ such that $\Delta N \subset N$. By a left $\Delta$-lattice $N \subset A$ we mean a left $\Delta$-module $N$ which is a 
$\Z$-lattice in $A$. 
We say that two left $\Delta$-ideals $M$ and $N$ are isomorphic if $M \cong N$ as left  $\Delta$-modules. 

If $\Lambda$ is a maximal  $\Z$-order in $A$, the set $LF_1(\Lambda)$ of isomorphism classes of 
left $\Lambda$-ideals in $A$ is a finite set, and its cardinality is independent of the choice of $\Lambda$ (see \cite[Thm. 26.4]{Re}). Therefore, we may define $h_A: = \vert LF_1(\Lambda) \vert $ for any maximal $\Z$-order $\Lambda$ in $A$; it is called the class number of the central simple $\Q$-algebra $A$.

\subsection{Integral structures for $SL_M$ - arithmetic groups} If $\Lambda_D$ is a $\Z$-order in $D$, then $\Lambda: = M_2(\Lambda_D)$ is a 
$\Z$-order in $M = M_2(D)$. Moreover, if the order $\Lambda_D$ is maximal, then $\Lambda$ is maximal as well (see \cite[Thm. 21.6]{Re}). In this situation, within the realm of group schemes of finite type defined over a Dedekind domain, we can carry through the  construction of affine $\Z$-group schemes $GL_{\Lambda}$ and $SL_{\Lambda}$ of finite type, similar to the construction of the algebraic $\Q$-groups $GL_M$ and $SL_M$. In fact, if $\Lambda$ is a maximal order,  the associated 
  affine $\Z$-group scheme $SL_{\Lambda}$ of finite type  is smooth (cf. \cite[Sect. 8.3]{Schbook}). The algebraic $\Q$-group $SL_{\Lambda} \times_{\Z} \Q$ obtained by base change is the group $G: = SL_M$ defined above. Occasionally we write $G_{\Lambda}$ for the affine group scheme $SL_{\Lambda}$, and  $\Gamma_{\Lambda}: = SL_{\Lambda}(\Z)$ for the group of integral points of  $G_{\Lambda}$. Consequently, in this notation, $G_{\Lambda} \times_{\Z} \Q = G$.
  Any subgroup $\Gamma$ of finite index in $\Gamma_{\Lambda}$ is an arithmetic subgroup of $G(\Q)$.
  
 Given any proper ideal 
$\mathfrak{a} \subset \Z$  the corresponding  principal congruence subgroup of level $\mathfrak{a}$ is defined by
\begin{equation}
\Gamma_{\Lambda}(\mathfrak{a}): = \mathrm{ker}(SL_{\Lambda}(\Z) \longrightarrow SL_{\Lambda}(\Z/\mathfrak{a})).
\end{equation}
It gives rise to an arithmetic subgroup $\Gamma_{\Lambda}(\mathfrak{a})$ of $G(\Q)$.  Using \cite[Prop. 4.4.4]{Schbook}, one deduces that for almost all choices of the ideal $\mathfrak{a}$ the group $\Gamma_{\Lambda}(\mathfrak{a})$ is torsion-free.

Given a prime ideal $\frak{p} = (p)$ in $\Z$, $p$ a prime, and $j \geq 1$, let $\Gamma_{\Lambda}(\mathfrak{p^j})$ be the corresponding congruence subgroup of level $\frak{p}^j$. The collection 
$\{\Gamma_{\Lambda}(\mathfrak{p^j})\}$ forms a cofinal tower of normal subgroups of $\Gamma_{\Lambda} = SL_{\Lambda}(\Z)$, to be called $\frak{p}$-congruence tower. In other words, we have 
\begin{equation}
\bigcap_{j} \Gamma_{\Lambda}(\mathfrak{p^j}) = \{1 \}, \; \Gamma_{\Lambda}(\mathfrak{p^j}) \; \text{is of finite index in} \; SL_{\Lambda}(\Z), \; \text{and}
 \; \Gamma_{\Lambda}(\mathfrak{p^{j+1}}) \leq \Gamma_{\Lambda}(\mathfrak{p^j}) \; \text{for all}\; j.
  \end{equation}
We note that for any $j \geq 1$, the quotient group $\Gamma_{\Lambda}(\mathfrak{p^{j}})/\Gamma_{\Lambda}(\mathfrak{p^{j+1}})$ is an elementary abelian $p$-group. 

\subsection{The ambient Lie group}\label{ambi}
We denote by $G_{\infty}$ the  group of real points of the algebraic $\Q$-group $SL_M$ attached to the central simple  $\Q$-algebra $M_2(D)$ over a central simple $\Q$-algebra $D$ of degree two.
Let $X_{G}$ be the symmetric space associated with $G_{\infty}$, described as the space of maximal compact subgroups of $G_{\infty}$. In fact, all of these are conjugate to one another, thus, we may write $X_{G} = K_{\infty}\backslash G_{\infty}$ for any maximal subgroup $K_{\infty} \subset G_{\infty}$. Since $X_{G}$ is diffeomorphic to $\R^{d(G_{\infty})}$ where $d(G) = \dim G_{\infty} - \dim K_{\infty}$, the space $X_{G}$ is contractible.   The real Lie group $G_{\infty}$ acts properly from the right on $X_G$, therefore a given arithmetic subgroup $\Gamma$ of $G(\Q)$, being  viewed as a discrete, thus closed subgroup of $G_{\infty}$,
 acts properly on $X_G$ as well. If $\Gamma$ is torsion-free, the action of $\Gamma$ on $X_G$ is free, and the quotient $X_G/ \Gamma$ is a smooth manifold of dimension $d(G)$.

There is a $G_{\infty}$-invariant Riemannian metric on $X_{G}$.
Given an arithmetic  subgroup $\Gamma$ of $G(\Q)$, we are interested in the homogenous space $X_{G}/\Gamma$. If $\Gamma$ is torsion-free, the space $X_{G}/\Gamma$ carries the structure of a Riemannian manifold of finite volume.

Depending on the splitting behaviour of the central simple $\Q$-algebra $D$ at the archimedean place of $\Q$, we have to distinguish two cases: 

{$D$ \it splits at infinity:} If $D$ splits at the archimedean place, we have $D \otimes_{\Q} \R \cong M_2(\R)$, thus, $G_{\infty} \cong SL_4(\R)$. The corresponding symmetric space $X_G$ has dimension  
$d(G) = 9$.

{$D$ \it ramifies at infinity:} If $D$ ramifies at the only archimedean place, that is, $D$ is a definite central simple $\Q$-algebra
we have $D \otimes_{\Q} \R \cong \mathbb{H}$. Therefore, $M: = M_2(\mathbb{H})$ is the central simple algebra of $(2 \times 2)$-matrices with entries in the algebra of Hamilton quaternions.
The corresponding ambient Lie group, $SL_M(\R)$, 
the special linear group $SL_2$ over the non-commutative algebra $\mathbb{H}$ of Hamilton quaternions, is usually denoted by $SU^*(4)$. This group 
is the real form of $SL_{4}(\C)$ associated with the complex conjugation $\sigma: SL_{4}(\C) \longrightarrow SL_{4}(\C)$, defined by 
  $g \mapsto \eta_2^t \overline{g}  \eta_2$, where
  $\eta_2 =   \bigl ( \begin{smallmatrix} 0 & E_2 \\ - E_2 & 0 \end{smallmatrix} \bigr )$, with $E_2$ the identity matrix of size $2$, and where $\overline{g}$ stands for conjugating each entry of the matrix $g$.

  The assignment 
$(x_1, x_2, x_{3}, x_{4}) \mapsto (\overline{x}_{3}, \overline{x}_{4}, - \overline{x}_1, \ \overline{x}_2)$ defines a real linear map $\psi: \C^{4} \longrightarrow \C^{4}$.
Then we can realise the  real Lie group $SU^*(4)$   as
$\{g \in SL_{4}(\C)\; \vert \; g \psi = \psi g \}$. Its intersection with the maximal compact subgroup $U(4)$ of $GL_{4}(\C)$ is the group $Sp(2) = \{g \in SU^*(4)\;\vert \; g \overline{g}^t = \overline{g}^t g = \mathrm{1} \}$.
(see \cite[X, Lemma 2.1]{Hel}). The symmetric space $X_2: = Sp(2)\backslash SU^*(4)$, attached to the Riemannian symmetric pair $(SU^*(4), Sp(2))$ of non-compact type is of type  A II. It is a simply connected space of dimension $d(G)  = 5$ and of rank one.
Since the symmetric space $X_2$ of type AII  coincides with the symmetric space of type BD I attached to the pair $(SO(5, 1)_0, SO(5) \times SO(1))$, $X_2$ can be identified with hyperbolic $5$-space, to be denoted $\mathrm{H}^5$.

\subsection{Strong approximation property}\label{strongapp}  Since  the $\Q$-group $G = G_{\Lambda} \times_{\Z} \Q$, associated with a maximal order $\Lambda$ in $M_2(D)$, is $\Q$-simple simply connected and $G_{\infty}$ is not compact, the group $G$ has the strong approximation property (see \cite{Kneser}). Therefore, $G(\Q)$ is dense in the locally compact group $G(\A_{\Q, f})$, or, equivalently, $G_{\infty} G(\Q)$ is dense in $G(\A_{\Q})$.

Given a  place $v \in V_{\Q}$, we denote by $k_v = \Q_v$ the completion of $\Q$ with respect to the normalised absolute value which corresponds to $v$. For a finite place $v \in V_{k, f}$, the field $k_v$ is a local field, and we write $\mathcal{O}_v$ for the valuation ring in $k_v$. It contains the valuation ideal $\frak{p}_v$ as the unique maximal ideal.

For each $v \in V_{\Q, f}$, we write $K_v: = G_{\Lambda}(\mathcal{O}_v)$. Then the group $\Gamma_{\Lambda} = G_{\Lambda}(\Z)$ of integral points can be written as 
$\Gamma_{\Lambda} =  G(\Q) \cap \prod_{v \in V_{\Q, f}} K_v$. For the sake of simplicity, we put $K_f: = \prod_{v \in V_{\Q, f}} K_v$.
For any $v \in V_{\Q, f}$, given any power $\frak{p}^e_v$ of the valuation ideal, we denote by
$$K_v(\frak{p}^e_v): = \mathrm{ker}(G_{\Lambda}(\mathcal{O}_v) \longrightarrow G_{\Lambda}(\mathcal{O}_v/\frak{p}^e_v\mathcal{O}_v))$$
the kernel of the natural homomorphism obtained by restriction. It is an open compact subgroup of $G_{\Lambda}(\mathcal{O}_v)$.
Any  ideal $\frak{a}$ has a  factorisation, unique up to the order of factors,
$\frak{a} = \prod \frak{p}_v^{e_v(\frak{a})},$
where the product extends over the finite set of all places $v \in V_{\Q, f}$ for which $\frak{p}_v$ occurs with the exponent $e_v(\frak{a})$ in this factorisation of $\frak{a}$.
For a fixed place $v \in V_{\Q, f}$, we have  $\frak{a} \mathcal{O}_v = \frak{p}_{v}^{e_v(\frak{a})} \mathcal{O}_v$. 
The direct product 
\begin{equation}
K_f(\frak{a}): = \prod_{\substack{v \in V_{\Q, f}\\ \frak{p}_v \vert \frak{a}}} K_v(\frak{p}_v^{e_{v}(\frak{a})}) \times 
\prod_{\substack{v \in V_{\Q, f}\\ \frak{p}_v \nmid \frak{a}}} K_v
\end{equation}
 is an open compact subgroup of $G(\A_{\Q, f})$.  We see
$\Gamma_{\Lambda}(\frak{a}) = G(\Q) \cap  K_f(\frak{a}).$
Using the strong approximation property of the algebraic $\Q$-group $G$, the continuous map $G_{\infty} \longrightarrow K_f(\frak{a}) \backslash G(\A_{\Q})/G(\Q)$, defined by 
the assignment 
$g \mapsto K_f(\frak{a})g G(\Q)$, induces 
a homeomorphism $$K_f(\frak{a})\backslash G(\A_{\Q})/G(\Q) \tilde{\longrightarrow} G_{\infty}/\Gamma_{\Lambda}(\frak{a})$$ which is equivariant under the action of $G_{\infty}$.

\section{Reduction theory - a suitable compactification}
We retain our previous notation and consider
a central simple  $\Q$-algebra $D$ of degree two and the algebraic $\Q$-group $G = SL_{\Lambda} \times_{\Z} \Q$ attached to a maximal order $\Lambda$ in the central simple $\Q$-algebra $M= M_2(D)$. 

\subsection{A result in reduction theory}
The $\Q$-rank of $G$ is one, thus, all proper parabolic  $\Q$-subgroups of $G$ are minimal, and all of these are conjugate under $G(\Q)$. Given an arithmetic subgroup $\Gamma$ of $G(\Q)$, the only $G(\Q)$-conjugacy class falls into finitely many $\Gamma$-conjugacy classes (see \cite[Prop. 15.6]{Bo1}).

As an application of the main results in reduction theory (see \cite[Thms. 1.2.2 and 1.2.3]{Ha4}), we have, in the specific case of an algebraic $\Q$-group of $\Q$-rank one, the following result:

\begin{theorem}\label{reduction} Let $D$ be  a central simple  $\Q$-algebra of degree two, and let  $G = SL_{\Lambda} \times_{\Z} \Q$ be the  algebraic $\Q$-group attached to a maximal order $\Lambda$ in the central simple $\Q$-algebra $M= M_2(D)$. 
Given a torsion-free arithmetic subgroup $\Gamma \subset SL_{\Lambda}(\Z) \subset G(\Q)$, the  locally symmetric space $X_G/\Gamma$ contains an open subset $Y_{\Gamma} \subset X_G/\Gamma$ such that its closure $\overline{Y}_{\Gamma}$ is a compact manifold with boundary $\partial \overline{Y}_{\Gamma}$, and the inclusion $\overline{Y}_{\Gamma} \longrightarrow X_G/\Gamma$ is a homotopy equivalence. 
The connected components of  $\partial \overline{Y}_{\Gamma}$ 
are parametrized by  the finite set $\mathcal{P}/\Gamma$ of 
$\Gamma$-conjugacy classes of minimal parabolic $\Q$-subgroups of $G$.
We have as a disjoint union 
$\partial \overline{Y}_{\Gamma} = \coprod_{[P] \in \mathcal{P}/\Gamma}Y^{[P]}.$
\end{theorem}

\subsection{The boundary components} In a more general context, a description of  a boundary component $Y^{[P]}$ as a fibre bundle is given in \cite[Theorem 5.4]{KSch}. Though we need some technical paraphernalia,
in our specific case the result has a simple form.
Given the standard minimal parabolic $\Q$-subgroup $P_0 = L_0 N_0$ of $G$,
any $\Q$-character $\chi: L_0 \longrightarrow \G_m$ induces a homomorphism
$
\chi_{\infty}: L_{0, \infty} \longrightarrow \G_{m, \infty} \cong (\R^{\times})$.
The absolute value  $\vert \cdot \vert$  on $\R$ defines the norm homomorphism 
$\vert \cdot \vert:  \G_{m, \infty} \cong   (\R^{\times}) \longrightarrow \R^{\times}_{> 0}, \quad g \mapsto 
 \vert g \vert.
$
We canonically extend the compositum $\vert \cdot \vert \circ \chi_{\infty}: L_{0, \infty} \longrightarrow \R^{\times}_{> 0}$   to a homomorphism
$\vert \chi \vert: P_{0, \infty} \longrightarrow \R^{\times}_{> 0}.
$
We apply this construction to the character $\rho: L_0 \longrightarrow \G_m$, given by the assignment
$$L_0(\Q) = \left \{ g =  \bigl ( \begin{smallmatrix}
x & 0 \\
0 & x^{-1} \end{smallmatrix} \bigr ) \; \vert \; x  \in
D^{\times} \right \} \longrightarrow \G_m(\Q), \quad \bigl ( \begin{smallmatrix}
x & 0 \\
0 & x^{-1} \end{smallmatrix} \bigr ) \mapsto \mathrm{nrd}_{D/\Q}(x).$$
We define 
\begin{equation}\label{Lnull}
L^{(1)}_{0, \infty} = \left \{ g =  \bigl ( \begin{smallmatrix}
x & 0 \\
0 & x^{-1} \end{smallmatrix} \bigr ) \; \vert \; x  \in
D^{\times}_{\infty},  \;  \vert \mathrm{nrd}_{D/\Q}(x) \vert  = 1 \right \} \;\text{resp.}\;
P^{(1)}_{0, \infty}: = \{p \in P_{0, \infty}\; \vert \;  \vert \rho \vert (p) = 1 \}.
\end{equation}
Given any point $x \in X_G$,  let $K_x \subset G_{\infty}$ be the corresponding maximal compact subgroup of the Lie group $G_{\infty}$, then $P^{(1)}_{0, \infty} \cap K_x = P_{0, \infty} \cap K_x$.  Moreover, since the image of the arithmetic group $\Gamma$ under $\rho$ is an arithmetic subgroup of $\G_m(\Q)$, thus,  contained in $\{\pm 1\}$, we have $ \vert \rho \vert (\gamma) = 1$ for every $\gamma \in  P_{0, \infty} \cap \Gamma$. 
It follows that $P_{0, \infty} \cap \Gamma = P^{(1)}_{0, \infty} \cap \Gamma$.  Given any other minimal parabolic subgroup $P$ of $G$, there is a $g \in G(\Q)$ such that 
$gP(\Q)g^{-1} = P_0(\Q)$. Therefore, we can define $P^{(1)}_{\infty}$ via conjugation.

 Given the Levi decomposition $P = LN$ of the minimal parabolic $\Q$=subgroup $P$ as the semi-direct product
of its unipotent radical $N$ and the Levi subgroup $L$,
there is 
 a surjective morphism $p: P^{(1)}_{\infty} \longrightarrow L^{(1)}_{\infty}$. The preimage of  a point in $L^{(1)}_{\infty}$ is diffeomorphic to the group $N_{\infty}$ of real points of the commutative group $N$.
 
 The image $K_L: = p(K \cap P^{(1)}_{\infty})$ of $K \cap P^{(1)}_{\infty}$ under this projection is a maximal compact subgroup in $L^{(1)}_{\infty}$. We write $Z_L: = K_L \backslash L^{(1)}_{\infty}$ for the associated manifold of right cosets.
The image $\Gamma_L$ of $P^{(1)}_{\infty} \cap \Gamma$ under $p$ is a discrete torsion-free subgroup of $L^{(1)}_{\infty}$. The group $\Gamma_L$ acts properly and freely on $Z_L$, and the double coset space $Z_L/ \Gamma_L$ is a manifold with universal cover $Z_L$. The projection $p: P^{(1)}_{\infty} \longrightarrow L^{(1)}_{\infty}$ induces 
 a locally trivial fibration 
$
\pi: (K \cap \Gamma) \backslash P^{(1)}_{\infty} / (P^{(1)}_{\infty} \cap \Gamma) \longrightarrow Z_L/ \Gamma_L;
$
with  the compact manifold $N_{\infty}/(N_{\infty} \cap \Gamma)$ as fibre. In fact, $Y^{[P]}$, 
admits the structure of a fibre bundle  which 
is equivalent to the fibre bundle 
 \begin{equation}\label{fibreb}
 (Z_L \times_{\Gamma_L} N_{\infty}/(N_{\infty} \cap \Gamma), Z_L, N_{\infty}/(N_{\infty} \cap \Gamma)).
  \end{equation}
 This bundle  is associated by the natural action of $\Gamma_L$ on the compact fibre $N_{\infty}/(N_{\infty} \cap \Gamma)$, induced by inner automorphisms, to the universal covering  $Z_L \longrightarrow Z_L/ \Gamma_L$ (cf. \cite[Prop. 4.3]{KSch}).

\begin{proposition}\label{boundary}
Let $D$ be  a   central simple  $\Q$-algebra of degree two,  let  $G = SL_{\Lambda} \times_{\Z} \Q$ be the  algebraic $\Q$-group attached to a maximal order $\Lambda$ in the central simple $\Q$-algebra $M= M_2(D)$, and let $\Gamma \subset SL_{\Lambda}(\Z) \subset G(\Q)$ be a torsion-free arithmetic subgroup.
\begin{itemize}
\item[-]{If $D$ is a definite $\Q$-algebra, the  locally symmetric space $X_G/\Gamma \cong \mathrm{H}^{5}/\Gamma$ is homotopy equivalent to the compact $5$-dimensional   manifold
$\overline{Y}_{\Gamma}$ whose boundary $\partial \overline{Y}_{\Gamma}$ is a disjoint union of $4$-dimensional tori.}
\item[-]{If $D$ is an indefinite $\Q$-algebra, the  locally symmetric space $X_G/\Gamma$ is homotopy equivalent to the compact $9$-dimensional  manifold
$\overline{Y}_{\Gamma}$. A connected component $Y^{[P]}$ of its boundary $\partial \overline{Y}_{\Gamma}$ is diffeomorphic to the total space of a fibre bundle whose base space 
$Z_L/\Gamma_L $ is diffeomorphic to an arithmetically defined compact quotient of the product $\mathrm{H}^2 \times \mathrm{H}^2$ of two upper half planes 
and whose  compact fibres are equal to  $4$-dimensional tori.}
\end{itemize}
\end{proposition}
\begin{proof}
We may assume that $P$ is the standard minimal parabolic $\Q$-subgroup $P_0 = L_0 N_0$ whose group of $\Q$-points is
$
P_0(\Q) = \left \{ g =  \bigl ( \begin{smallmatrix}
x & y \\
0 & x^{-1} \end{smallmatrix} \bigr ) \; \vert \; x \in D^{\times}, y \in D
\right \}.
$ The group of $\Q$-points of its unipotent radical is commutative, and we obtain as additive groups
$N_{0, \infty} \cong (D_a)(\R) \cong \mathbb{H}$ in the first case, and $N_{0, \infty} \cong M_2(\R)$ in the second case. Therefore, it follows that $N_{0, \infty} \cong  \R^4$.
 The group $N_{0, \infty} \cap \Gamma$ as  a discrete subgroup of $N_{0, \infty}$ forms a complete lattice in $\R^{4}$, thus, the fibre of the bundle (\ref{fibreb}) is a $4$-dimensional torus.  
 
 In the case of a definite central simple $\Q$-algebra of degree two, we have, by the very definition, $D_{\infty} \cong \mathbb{H}$. Since 
 $ \mathbb{H}^{(1)} =\{h \in  \mathbb{H}\;\vert\; \vert \mathrm{nrd}_{\mathrm{H}/\Q}(h)\vert  = 1 \}$ is a compact group, it follows  that the base $Z_L$  of the fibre bundle is a point, and hence $Y^{[P]} \cong N_{0, \infty}/ (N_{0, \infty} \cap \Gamma)$.
 
 In the case of an indefinite central simple  $\Q$-algebra $D$, we have $D_{\infty} \cong M_2(\R)$.   In obvious notation, we see $D^{(1)}_{\infty} \cong \{m \in M_2(\R)^{\times}\; \vert \vert \mathrm{nrd}_{M_2(\R)/\R}(m) \vert = 1 \} \cong SL_2(\R)^{\pm}$.
 The corresponding symmetric space is the hyperbolic $2$-space $\mathrm{H}^2$. Therefore, we see $Z_L \cong  \mathrm{H}^2 \times \mathrm{H}^2$. Recall that the maximal order 
 $\Lambda = M_2(\Lambda_D)$ originates from a maximal order $\Lambda_D$ in $D$.  The group $\Gamma_L$ is a subgroup of finite index in the group
  \begin{equation*}
L^{(1)}_{0, \Lambda}: = \left \{ \ell =  \bigl ( \begin{smallmatrix}
\lambda& 0 \\
0 & \lambda^{-1} \end{smallmatrix} \bigr ) \; \vert \; \lambda  \in
\Lambda_D^{\times},  \;  \vert \mathrm{nrd}_{D/\Q}(\ell) \vert  = 1 \right \}. 
\end{equation*}
The group $L^{(1)}_{0, \Lambda}$ is an arithmetically defined subgroup of $L^{(1)}_{0, \infty}$ and acts component wise on $Z_L \cong  \mathrm{H}^2 \times \mathrm{H}^2$.
It follows that $Z_L/\Gamma_L$ is diffeomorphic to the quotient of $\mathrm{H}^2 \times \mathrm{H}^2$ obtained by the action of the group $\Gamma_{L, D} \times \Gamma_{L, D}$
 where $\Gamma_{L, D}$ is a subgroup of finite index in the group $\Lambda_{D}^{\times}$ of units in $\Lambda_D$. Using the theorem of K\"ate Hey (see \cite[Thm. 3.4.1]{Schbook}), the quotient $D^{(1)}_{\infty}/\Gamma_{L, D}$ is compact. 
 \end{proof}

\section{On the cohomology  $H^*(X_G/\Gamma, \C)$}\label{gencoh}
\subsection{The long exact sequence in cohomology attached to the pair $\overline{Y}_{\Gamma}, \partial \overline{Y}_{\Gamma}$}  Let $D$ be a central simple  $\Q$-algebra algebra of degree two. Given a maximal order $\Lambda_D$ in $D$, the corresponding order $\Lambda = M_2(\Lambda_D)$ defines the $\Z$-group scheme $SL_{\Lambda}$. Via base change, we obtain the algebraic $\Q$-group  $G = SL_{\Lambda} \times_{\Z} \Q = SL_{M_2(D)}$. Then 
$\Gamma_{\Lambda}: = SL_{\Lambda}(\Z)$ is an arithmetic subgroup of $G(\Q)$. Any subgroup of $\Gamma_{\Lambda}$ acts properly on the  symmetric space $X_G$ of dimension $d(G)$ attached to the group of real points of $G$.

 Given a torsion-free arithmetic subgroup $\Gamma \subset G(\Q)$, the quotient $X_G/\Gamma$ carries the structure of a complete Riemannian manifold of finite volume but non-compact. By Theorem \ref{reduction}, there exists an open subset $Y_{\Gamma} \subset X_G/\Gamma$ such that its closure $\overline{Y}_{\Gamma}$ is a compact manifold with boundary $\partial \overline{Y}_{\Gamma} $, and the inclusion $\overline{Y}_{\Gamma} \longrightarrow X_G/\Gamma$ is a homotopy equivalence. Therefore, in terms of singular or deRham cohomology, we have $H^q(\overline{Y}_{\Gamma}, \C) \cong H^q(X_G/\Gamma, \C)$.  Note that $H^q(X_G/\Gamma, \C) = \{0\}$ for any $q \geq d(G)$.
We have a long exact sequence in cohomology 
$$
\longrightarrow H^q(\overline{Y}_{\Gamma}, \partial \overline{Y}_{\Gamma}, \C) \longrightarrow 
H^q(\overline{Y}_{\Gamma}, \C)  \overset{r^q}{\longrightarrow} H^q(\partial \overline{Y}_{\Gamma}, \C) \longrightarrow  H^{q+1}(\overline{Y}_{\Gamma}, \partial \overline{Y}_{\Gamma}, \C) \longrightarrow
$$
where $r{^q}: H^q(\overline{Y}_{\Gamma}, \C)  \overset{r^q}{\longrightarrow} H^q(\partial \overline{Y}_{\Gamma}, \C)$ denotes the natural homomorphism obtained by restriction.
  
 \begin{proposition}\label{image}  Let 
 $\sR^q: = \mathrm{im}(H^q(\overline{Y}_{\Gamma}, \C)  \overset{r^q}{\longrightarrow} H^q(\partial \overline{Y}_{\Gamma}, \C))$ denote the image of the restriction homomorphism $r^q$. Then 
 $\dim \sR^q + \dim \sR^{d(G) - 1 - q} = \dim H^q(\partial \overline{Y}_{\Gamma}, \C)$. Therefore, it follows that the  image of $r^{*}$
  has  dimension equal to one half of the dimension of $H^*(\partial \overline{Y}_{\Gamma}, \C)$. 
\end{proposition}
\begin{proof}
As pointed out in \ref{ambi}, 
depending on the splitting behaviour of the central simple $\Q$-algebra $D$ at the archimedean place of $\Q$, we have to distinguish two cases: 
Firstly,   $D$ splits at the archimedean place, that is,  $D \otimes_{\Q} \R \cong M_2(\R)$, and the corresponding symmetric space $X_G$ has dimension  
$d(G) = 9$. Secondly, 
$D$ ramifies at the archimedean place 
that is, $D \otimes_{\Q} \R \cong \mathbb{H}$, and the corresponding symmetric space $X_G$ is the $5$-dimensional hyperbolic space, thus, $d(g) = 5$. 
In both cases we are within the range of oriented compact manifolds with boundary of odd dimension $2k+1$.

The actual proof relies on a general result in the cohomology theory of oriented compact manifolds with boundary, with a focus on duality. Due to the lack of a suitable reference we include it in the following Proposition.
\end{proof}

 \begin{proposition} Let $M$ be an oriented compact manifold of odd dimension $2k+1$ with boundary $\partial M$, and let $F$ be any field of characteristic zero.  The image of the homomorphism
 $$r^*: H^*(M, F) \longrightarrow H^*(\partial M, F),$$
 induced by the inclusion $r: \partial M \longrightarrow M$, is of dimension half of the dimension of $H^*(\partial M, F)$. More presicely, 
 if we denote $R^j: = \mathrm{im}( r^j: H^j(M, F) \longrightarrow H^j(\partial M, F))$, then 
 $$\dim R^j + \dim R^{2k-j} = \dim H^j(\partial M, F) \; \text{for all}\;  j.$$
 In particular, we have $\dim R^k = (1/2) \dim H^k(\partial M, F)$.
  \end{proposition}
 \begin{proof}
 Given the oriented compact manifold $M$ with boundary $\partial M$ of dimension $m = 2k+1$, there is the uniquely determined fundamental class $O_M$ in
 $H_m(M, \partial M; \Z)$. The boundary operator in homology takes $O_M$ to a fundamental class $o_{\partial M}$ of $\partial M$ in $H_{m-1}(\partial M;\Z)$ (see, e.g. \cite[Chap.6, Sec.3, Cor. 10]{Span}).
 
 Now suppose that $N$ is an oriented compact manifold. Then the cap product with a fundamental class of $N$, that is, the assignment $x \mapsto x \frown o_N$, defines an isomorphism
 $H^i(N; F) \tilde{\longrightarrow} H_{n-i}(N; F)$ where the coefficient system $F$ is any field of characteristic zero.
 
 More generally, in the case of the oriented compact manifold $M$ with boundary, then the map $x \mapsto x \frown O_M$ defines an isomorphism
 $H^i(M, \partial M; F) \tilde{\longrightarrow} H_{n-i}(M; F)$ (see \cite[Chap. 6, sec. 3, Thm. 12]{Span}). In addition, the pairing
 $$H^i(M, \partial M;F) \times H^{m-i}(M; F) \longrightarrow H^m(M, \partial M; F) \longrightarrow H_0(M; F) = F,$$
 defined by the assignment
 $$(x, y) \mapsto x \smile y \mapsto \langle x \smile y, O_M \rangle = \langle x, y \frown O_M \rangle,$$
 where $(x, y) \mapsto x \smile y $ denotes the cup product, and $\langle \; , \; \rangle$ the Kronecker product, is a dual pairing.
 
 The pair $(M, \partial M)$, with the inclusion $r: \partial M \longrightarrow M$, gives rise to a long exact sequence in cohomology as well as in homology (with coefficients in $F$). Therefore, taking into account duality for manifolds with boundary, 
 the following diagram with exact rows is commutative, and all vertical arrows are isomorphisms (up to sign):
 
 \begin{equation*}
\xymatrix{
  \ar[r] & H^{m-i-1}(M)\ar[d]^-{\frown O_M} \ar[r]^-{r^*} & H^{m-i-1}(\partial M) \ar[d]^-{\frown o_{\partial M}} \ar[r]^-{\delta^*} & H^{m-i}(M, \partial M) \ar[d]^-{\frown O_M} \ar[r]  & H^{m-i}(M) \ar[d]^-{\frown O_M} \ar[r] & \\
  \ar[r] & H_{i+1}(M, \partial M) \ar[r]^-{\delta_*} & H_i(\partial M)  \ar[r]^-{r_*} & H_{i}(M)  \ar[r] & H_i(M, \partial M) \ar[r] &.}
\end{equation*}
Given any $x \in R^* =  \mathrm{im}( r^*: H^*(M, F) \longrightarrow H^*(\partial M, F))$, we have $\delta^*(x) = 0$, and therefore $r_*(x \frown o_{\partial M}) = 0$ holds. It follows that $\langle H^*(M), r_*(x \frown o_{\partial M})\rangle = \{0\}$. Using the fact that $r_*$ and $r^*$ are transposed morphisms with respect to the Kronecker product, we get $\langle R^*, x \frown  o_{\partial M}\rangle = \{0\}$. By the duality between the cap-product and the cup-product, we see $\langle R^* \smile x,  o_{\partial M}\rangle = \{0\}$. It follows that, with respect to the dual pairing defined by $(x, y) \mapsto \langle x \smile y,  o_{\partial M}\rangle$, the space $R^*$ is a maximal isotropic subspace of $H^*(\partial M)$. Taking into account the cohomological degrees in question and using standard linear algebra, this implies the assertions.
  \end{proof}
  
 \begin{remark}\label{codim}
 We  retain the  notation in Proposition \ref{image}. By Poincare duality, we have $H^{d(G}(\overline{Y}_{\Gamma}, \partial \overline{Y}_{\Gamma}, \C) \cong H_0(\overline{Y}_{\Gamma}, \C)$. Since the latter space is one-dimensional, the image $\sR^{d(G)-1}$ of $r^{d(G)-1}$ in $H^{d(G)-1}(\partial \overline{Y}_{\Gamma}, \C)$ has codimension one.
 \end{remark}
 
\section{The size of the cohomology at infinity}
\subsection{The set  of $\Gamma$-conjugacy classes of minimal parabolic $\Q$-subgroups of $G$}
 Given a torsion-free arithmetic subgroup $\Gamma \subset G(\Q)$, the connected components of the  boundary $\partial \overline{Y}_{\Gamma} $ of the compact manifold $\overline{Y}_{\Gamma}$ are parametrised by the set  of $\Gamma$-conjugacy classes of minimal parabolic $\Q$-subgroups of $G$. Since all proper parabolic $\Q$-subgroups of $G$ are conjugate under $G(\Q)$ to the standard minimal parabolic $\Q$-subgroup $P_0$, this set is described by the set $\Gamma \backslash G(\Q) /P_0(\Q)$ of double cosets. The number of double cosets with respect to $(\Gamma, P_0(\Q))$ is finite (see \cite[Prop. 15. 6]{Bo1}).

 Using the notation of Section \ref{schemes}, in particular, \ref{strongapp}, 
 the group $\Gamma_{\Lambda} = G_{\Lambda}(\Z)$ of integral points of the group scheme $G_{\Lambda}$ can be written as 
$\Gamma_{\Lambda} =  G(\Q) \cap K_f$ where $K_f = \prod_{v \in V_{\Q, f}} K_v$ with $K_v = G_{\Lambda}(\mathcal{O}_v)$. Accordingly, for the principal congruence subgroup $\Gamma_{\Lambda}(\frak{a}) $, attached to a proper ideal $\frak{a}$, we have 
$\Gamma_{\Lambda}(\frak{a}) = G(\Q) \cap  K_f(\frak{a}).$ 
Since the group $G_{\Lambda} \times \Q$ has the strong approximation property, the assignment 
$(G(\Q) \cap K_f) g P_0(\Q) \mapsto K_f g P_0(\Q)$ gives rise to a bijection
\begin{equation}
\Gamma_{\Lambda} \backslash G(\Q) / P_0(\Q) = (G(\Q) \cap K_f) \backslash G(\Q) / P_0(\Q) \tilde{\longrightarrow} K_f \backslash G(\A_f) / P_0(\Q).
\end{equation} 
Given a finite place $v \in V_{k, f}$, we have the Iwasawa decomposition $G(k_v) = G_{\Lambda}(\mathcal{O}_v) P_0(k_v)$, therefore, $G(\A_f) = K_f P_0(\A_f)$. It follows that the representatives of the double cosets $K_f \backslash G(\A_f) / P_0(\Q)$ can be chosen in $P_0(\A_f)$.
Indeed, we may identify  (see e.g. \cite[Prop. 7.5]{Bo1})
\begin{equation}\label{ident}
K_f \backslash G(\A_f) / P_0(\Q) \tilde{\longrightarrow} (K_f \cap P_0(\A_f)) \backslash P_0(\A_f) / P_0(\Q).
\end{equation}
We note that the cardinality of the latter set is the class number of $P_0$ with respect to $K_f \cap P_0(\Q)$. Taking into account the Levi decomposition $P_0 = L_0 N_0$ of $P_0$, this class number is equal to the class number of $L_0$ with respect to $K_f \cap L_0(\Q)$.  Using \cite[Prop.2.4]{Bo1}, this follows from the fact that the class number of the unipotent radical is one (see \cite[Cor. 2.5]{Bo1}.

Using these reduction steps to a question regarding centraliser of algebraic $\Q$-tori one has derived the following result:
  
  \begin{proposition}\label{cusps}
  Let $D$ be  a   central simple  division $\Q$-algebra of degree two, and let $h_D$ be the class number of $D$ (see Section \ref{orders}). Let  $G = SL_{\Lambda} \times_{\Z} \Q$ be the  algebraic $\Q$-group attached to a maximal order $\Lambda$ in the central simple $\Q$-algebra $M= M_2(D)$. Then  the number of 
  $\Gamma_{\Lambda}$- conjugacy classes of minimal parabolic $\Q$-subgroups of $G$ is given as follows:
\begin{itemize}
\item[-]{If $D$ is a definite $\Q$-algebra, we have $\vert \Gamma_{\Lambda} \backslash G(\Q) / P_0(\Q) \vert  = h^2_D$.}
\item[-]{If $D$ is an indefinite $\Q$-algebra, we have $\vert \Gamma_{\Lambda} \backslash G(\Q) / P_0(\Q) \vert  = h_D$.}
\end{itemize}
In both cases,  this quantity is independent of the choice of the maximal order $\Lambda$.
\end{proposition}
\begin{proof}
In the case of a  definite $\Q$-algebra $D$, we refer to \cite[Satz 2.1]{KOS}), or, in a more general context, to \cite[Lemma 5]{Koch}. 
In the case of an indefinite  division $\Q$-algebra,  the assertion is shown in \cite[Chap. 5]{Lacher} or, with a different proof, in \cite[Thm. 4]{Koch}. 
\end{proof}
  
  Next, given a  principal congruence subgroup $ \Gamma_{\Lambda}(\frak{p})$ of $G(\Q)$, $\frak{p}$ a prime ideal at which $D$ splits,
we intend to determine the number $\mathrm{cusp}(G, \Gamma_{\Lambda}(\frak{p})): = \vert \Gamma_{\Lambda}(\frak{p}) \backslash G(\Q) / P_0(\Q) \vert$ of  $\Gamma_{\Lambda}(\frak{p})$- conjugacy classes of minimal parabolic $\Q$-subgroups of $G$, that is, the number of  connected components of the  boundary $\partial \overline{Y}_{\Gamma_{\Lambda}(\frak{p})} $ of the compact manifold $\overline{Y}_{\Gamma_{\Lambda}(\frak{p})}$.

   \begin{proposition}\label{cusps}
  Let $D$ be  a   central simple  division $\Q$-algebra of degree two, and let $h_D$ be the class number of $D$. Let  $G = SL_{\Lambda} \times_{\Z} \Q$ be the  algebraic $\Q$-group attached to a maximal order $\Lambda = M_2(\Lambda_D)$ in the central simple $\Q$-algebra $M= M_2(D)$. Then, given a  principal congruence subgroup $ \Gamma_{\Lambda}(\frak{p}^e)$ of $G(\Q)$, $\frak{p}^e$ a power of a prime ideal at which $D$ splits,

\begin{itemize}
\item[-]{if $D$ is a definite $\Q$-algebra, we have $\mathrm{cusp}(G, \Gamma_{\Lambda}(\frak{p}^e))  = h^2_D N(\frak{p^e})^{-4} [\Gamma_{\Lambda} : \Gamma_{\Lambda}(\frak{p^e})]$.}
\item[-]{if $D$ is an indefinite $\Q$-algebra, we have $\mathrm{cusp}(G, \Gamma_{\Lambda}(\frak{p}^e))  = h_D \mu_{\Lambda_D}(\frak{p}^e)^{-1} N(\frak{p}^e)^{-4} [\Gamma_{\Lambda} : \Gamma_{\Lambda}(\frak{p}^e)]$.}
\end{itemize}
where $\mu_{\Lambda_D}(\frak{p}^e)$ denotes the cardinality of the finite group  $\Lambda^{\times}_D/\Lambda^{\times}_D(\frak{p}^e)$, i.e. the units in $\Lambda_D$  modulo the elements which are congruent to one modulo $\frak{p}^e$.
\end{proposition}
\begin{proof} Using Proposition \ref{cusps}, we have to pass from the level given by the  group $\Gamma_{\Lambda}$ to the level given by the principal congruence subgroup $\Gamma_{\Lambda}(\frak{p}^e)$. Therefore we have to compare the set of double cosets $K_f(\frak{p}^e) \backslash G(\A_f) / P_0(\Q)$ with $K_f \backslash G(\A_f) / P_0(\Q).$. On the former set the group $K_f/ K_f(\frak{p}^e)$ acts via left multiplication, and the set of orbits is $K_f \backslash G(\A_f) / P_0(\Q).$ However, we still have to determine the stabiliser of a given double coset.
We know by our previous discussion, in particular, see identification (\ref{ident}), that the classes in $K_f \backslash G(\A_f) / P_0(\Q).$ are determined by elements in $L_0$, and that the unipotent radical $N_0$ has strong approximation. Therefore, if $D$ is a definite $\Q$-algebra, 
the stabilisers are isomorphic to  the additive group $\Lambda_D/\frak{p}^e \Lambda_D$. 
If $D$ is an indefinite $\Q$-algebra,
the stabilisers are isomorphic to the semi-direct product of the additive group $\Lambda_D/\frak{p}^e \Lambda_D$ and the finite group $\Lambda^{\times}_D/\Lambda^{\times}_D(\frak{p}^e)$, that is,  the units in $\Lambda_D$  modulo the elements which are congruent to one modulo $\frak{p}^e$.
\end{proof}

 Since the affine $\Z$-group scheme $G_{\Lambda} = SL_{\Lambda}$ of finite type associated with the maximal order $\Lambda$ in $M = M_2(D)$ is smooth, the strong approximation property of $G = G_{\Lambda} \times_{\Z} \Q$ implies that there is a short exact sequence of groups
\begin{equation}
1 \longrightarrow \Gamma_{\Lambda}(\frak{a})  \longrightarrow SL_{\Lambda}(\Z) = \Gamma_{\Lambda}   \longrightarrow SL_{\Lambda}(\Z/\frak{a})  \longrightarrow 1.
\end{equation}
This allows us to determine the index of the principal congruence subgroup  $\Gamma_{\Lambda}(\frak{a})$ in $SL_{\Lambda}(\Z)$. It is given by the cardinality of $SL_{\Lambda}(\Z/\frak{a})$. By means  of the strong approximation property of the group $G$, we may reduce this question to a local problem. Using the notation in Section \ref{strongapp}, we 
have
$\Gamma_{\Lambda} =  G(k) \cap \prod_{v \in V_{\Q, f}} K_v$ where $K_v = G_{\Lambda}(\mathcal{O}_v)$. 
If the ideal $\frak{a}$ has the factorisation, unique up to order, $\frak{a} = \prod \frak{p}_v^{e_v(\frak{a})},$ we have 
$\Gamma_{\Lambda}(\frak{a}) = K(\frak{a})$ where $K(\frak{a}): = G(k) \cap \prod_{v \in V_{\Q, f}} K_v(\frak{p}_v^{e_{v}(\frak{a})}).$ 
\begin{lemma}\label{power}
Given a principal congruence subgroup $\Gamma_{\Lambda}(\frak{a})$ attached to an ideal $\frak{a} = \prod_{v \in V_{\Q, f}} \frak{p}_v^{e_v(\frak{a})}$, we have
\begin{equation}
[\Gamma_{\Lambda} : \Gamma_{\Lambda}( \frak{a})] = \prod_{v \in V_{\Q, f}} [K_v : K_v(\frak{p}_v^{e_{v}(\frak{a})})].
\end{equation}
For any fixed place $v \in V_{\Q, f}$ we have $[K_v : K_v(\frak{p}_v^{e_{v}(\frak{a})})] = q_v^{(e_{v}(\frak{a})-1) \dim G} \vert  G_{\Lambda}(\mathcal{O}_v/\frak{p}_v \mathcal{O}_v) \vert$
where $\dim G$ denotes the dimension of $G$. We have $\dim G = 15$.
\end{lemma}
\begin{proof} Firstly, given the two subgroups $\Gamma_{\Lambda} =  G(k) \cap \prod_{v \in V_{\Q, f}} K_v$ and  
$K(\frak{a}) = \prod_{v \in V_{\Q, f}} K_v(\frak{p}_v^{e_{v}(\frak{a})})$ of $G(\A_{\Q, f})$ whereas the latter is normal,  there is a natural bijection between the two quotient groups
$$\Gamma_{\Lambda}/\Gamma_{\Lambda}(\frak{a}) = \Gamma_{\Lambda}/ \Gamma_{\Lambda} \cap K(\frak{a})  \tilde{\longrightarrow} (\Gamma_{\Lambda} K(\frak{a}))/K(\frak{a}) = 
(G(k) \cap \prod_{v \in V_{\Q, f}} K_v) K(\frak{a}) / K(\frak{a}).$$
As a consequence of the strong approximation property, we have $G(\A_{\Q, f}) = G(k) K(\frak{a})$. 
Therefore, the claim follows.

Secondly, since the group $G$ is smooth, the morphism $K_v = G_{\Lambda}(\mathcal{O}_v) \longrightarrow G_{\Lambda}(\mathcal{O}_v/\frak{p}^e_v\mathcal{O}_v))$ is surjective for any exponent $e$. By definition, its kernel is $K_v(\frak{p}^e_v)$. Therefore, for $e = 1$, the index $[K_v : K_v(\frak{p}_v)]$ is equal to the order of
 $G_{\Lambda}(\mathcal{O}_v/\frak{p}_v \mathcal{O}_v)$. Note that $\mathcal{O}_v/\frak{p}_v \mathcal{O}_v$ is the residue class field, say, of order $q_v = p_v^f$, $p_v$ a prime,  of the local field $k_v$. For every exponent $r \geq 1$, we have the isomorphism $\frak{p}_v^r/ \frak{p}_v^{r+1} \tilde{\longrightarrow} \mathcal{O}_v/\frak{p}_v$. Therefore, 
 we obtain via induction over the exponent $e$, that 
 $[K_v : K_v(\frak{p}_v^{e_{v}(\frak{a})})] = q_v^{(e_{v}(\frak{a})-1) \dim G} \vert  G_{\Lambda}(\mathcal{O}_v/\frak{p}_v \mathcal{O}_v) \vert.$
\end{proof}

Next we determine the (local) index  $[K_v : K_v(\frak{p}_v)] = \vert G_{\Lambda}(\mathcal{O}_v/\frak{p}_v \mathcal{O}_v) \vert.$ We have to distinguish two cases:

\begin{lemma}\label{index}

  Let $D$ be  a   central simple  division $\Q$-algebra of degree two, and let   $G_{\Lambda}  = SL_{\Lambda}$ be the  affine smooth $\Z$-group scheme of finite type  attached to a maximal order $\Lambda$ in the central simple $\Q$-algebra $M= M_2(D)$. Let $v \in V_{\Q, f}$ be a non-archimedean place.
  \begin{itemize}
\item[-] {If  $D$ splits at the place $v$, then $\vert G_{\Lambda}(\mathcal{O}_v/\frak{p}_v \mathcal{O}_v) \vert = q_v^{\dim G} \prod_{j = 2}^4(1 - q_v^{-j}).$}
\item[-] {If $D$ ramifies at the place $v$, then $ \vert G_{\Lambda}(\mathcal{O}_v/\frak{p}_v \mathcal{O}_v) \vert =  q_v^{\dim G} (1 + q_v^{-1})(1 - q_v^{-4})$}
\end{itemize}
where $\dim G = 15$ is the dimension of $G$.
Passing over to the principal congruence subgroup $\Gamma_{\Lambda}(\frak{p}_v^{e_v})$ we obtain for the index $[\Gamma_{\Lambda} : \Gamma_{\Lambda}( \frak{p}_v^{e_v})]$
  \begin{itemize}
\item[-] {If  $D$ splits at the place $v$, then $[\Gamma_{\Lambda} : \Gamma_{\Lambda}( \frak{p}_v^{e_v})]  = q_v^{e_v \cdot \dim G}  \prod_{j = 2}^4(1 - q_v^{-j}).$}
\item[-] {If $D$ ramifies at the place $v$, then $[\Gamma_{\Lambda} : \Gamma_{\Lambda}( \frak{p}_v^{e_v})]  = q_v^{e_v \cdot \dim G} (1 + q_v^{-1})(1 - q_v^{-4})$}
\end{itemize}
\end{lemma}
\begin{proof}Firstly, if $D$ splits at $v$, then $G_{\Lambda} \times_{\Z} \mathcal{O}_v $ is isomorphic to the $\mathcal{O}_v$-group scheme $SL_4$. The cardinality of its group of points over the finite field $\mathcal{O}_v/\frak{p}_v \mathcal{O}_v$ of cardinality  $q_v = p_v^f$, $p_v$ a prime, is given by
$$  \vert G_{\Lambda}(\mathcal{O}_v/\frak{p}_v \mathcal{O}_v) \vert = q_v^{\frac{4(4-1)}{2}} \prod_{j = 2}^4(q_v^j - 1) = q_v^{15} \prod_{j = 2}^4(1 - q_v^{-j}).$$
If $D$ ramifies at the place $v$, then, as determined 
 in \cite[V.4.7]{Kionkethesis},
 $ \vert G_{\Lambda}(\mathcal{O}_v/\frak{p}_v \mathcal{O}_v) \vert = q_v^{15} (1 + q_v^{-1})(1 - q_v^{-4}).$
 
 The second assertion regarding the index $[\Gamma_{\Lambda} : \Gamma_{\Lambda}( \frak{p}_v^{e_v})]$ follows by using Lemma \ref{power}.
\end{proof}

\subsection{Cohomology at infinity} As before, given a definite quaternion $\Q$-algebra $D$,  let $G = G_{\Lambda} \times_{\Z} \Q$ be the algebraic $\Q$-group associated with a maximal order $\Lambda$ in the central simple $\Q$-algebra $M_2(D)$ of degree four.
The quotient $X_G/\Gamma$, attched to  given a torsion-free arithmetic subgroup $\Gamma \subset \Gamma_{\Lambda} = G_{\Lambda}(\Z)$, 
 is a complete Riemannian manifold of finite volume but non-compact. Its cohomology is reflected in  the  long exact sequence in cohomology 
$$
\longrightarrow H^q(\overline{Y}_{\Gamma}, \partial \overline{Y}_{\Gamma}, \C) \longrightarrow 
H^q(\overline{Y}_{\Gamma}, \C)  \overset{r^q}{\longrightarrow} H^q(\partial \overline{Y}_{\Gamma}, \C) \longrightarrow  H^{q+1}(\overline{Y}_{\Gamma}, \partial \overline{Y}_{\Gamma}, \C) \longrightarrow
$$
 associated with the pair  $(\overline{Y}_{\Gamma}, \partial \overline{Y}_{\Gamma})$, see Theorem \ref{reduction} resp. Section \ref{gencoh}.
 We intend to determine the size of the image of the natural homomorphism
$r{^q}: H^q(\overline{Y}_{\Gamma}, \C)  \longrightarrow H^q(\partial \overline{Y}_{\Gamma}, \C)$ obtained by restriction.

In view of the isomorphism  $H^q(\overline{Y}_{\Gamma}, \C) \cong H^q(X_G/\Gamma, \C)$, the term $\dim \mathrm{im}(r^q)$ gives information about  the size of the ``cohomology of $\Gamma$ at infinity", that is, the space of classes in  $H^q(X_G/\Gamma, \C)$ which restrict non-trivially to 
$H^q(\partial \overline{Y}_{\Gamma}, \C)$.
In our discussion we only deal with the case of a definite $\Q$-algebra. We refrain from working out this topological approach in the case of an indefinite quaternion division $\Q$-algebra. In this case, the interpretation of $H^q(X_G/\Gamma, \C)$ in terms of automorphic forms, in particular, the use of Eisenstein series, gives better insight.
see Remark \ref{indef}.
 
 \begin{theorem}\label{final} Let $D$ be a definite quaternion $\Q$-algebra. Let $G = G_{\Lambda} \times_{\Z} \Q$ be the algebraic $\Q$-group associated with a maximal order $\Lambda$ in the central simple $\Q$-algebra $M_2(D)$.
 Given the principal congruence subgroup $\Gamma_{\Lambda}(\frak{p}^e) \subset \Gamma_{\Lambda} = G_{\Lambda}(\Z)$, $\frak{p}$ a prime which splits in $D$, 
 we write 
 $\sR^q(\Gamma_{\Lambda}(\frak{p}^e)): = \mathrm{im}(r^q: H^q(\overline{Y}_{\Gamma}, \C)  \longrightarrow H^q(\partial \overline{Y}_{\Gamma}, \C))$ for the  image of the map $r^q$. Then we have 
\begin{itemize}
\item[-] {$\dim \sR^0(\Gamma_{\Lambda}(\frak{p}^e)) = 1$, and $\sR^4(\Gamma_{\Lambda}(\frak{p}^e))$ is of codimension one in $H^*(\partial \overline{Y}_{\Gamma_{\Lambda}(\frak{p}^e)}, \C)$}
\item[-] {$\dim \sR^1(\Gamma_{\Lambda}(\frak{p}^e)) + \dim \sR^3(\Gamma_{\Lambda}(\frak{p}^e)) = \dim H^3(\partial \overline{Y}_{\Gamma_{\Lambda}(\frak{p}^e)}) =
4 h_D^2 N(\frak{p}^e)^{-4} [\Gamma_{\Lambda} : \Gamma_{\Lambda}(\frak{p}^e)]$}
\item[-] {$\dim \sR^2(\Gamma_{\Lambda}(\frak{p}^e))  = \frac{1}{2} \dim H^2(\partial \overline{Y}_{\Gamma_{\Lambda}(\frak{p}^e)}) =
3 h_D^2 N(\frak{p}^e)^{-4} [\Gamma_{\Lambda} : \Gamma_{\Lambda}(\frak{p}^e)]$}
\end{itemize}
 \end{theorem}
 \begin{proof} Using Proposition \ref{image}, the first assertion follows by remark \ref{codim}.  Next, the boundary $\partial \overline{Y}_{\Gamma_{\Lambda}(\frak{p}^e)}$
 is a disjoint union of $4$-dimensional tori, and the cohomology $H^*(T^4, \C)$ is isomorphic to the exterior algebra $\bigwedge^* \C^4$. Then, again using Proposition \ref{image}, the other assertions follow from Proposition \ref{cusps}.
 \end{proof}
 
 \begin{remark}In the expressions  given above for $\dim \sR^q(\Gamma_{\Lambda}(\frak{p}^e))$ one may replace the index $[\Gamma_{\Lambda} : \Gamma_{\Lambda}(\frak{p}^e)]$ by the formulae as given in Lemma \ref{index}. This allows us to give (up to a positive constant) the weaker bound $[\Gamma_{\Lambda} : \Gamma_{\Lambda}(\frak{p}^e)]^{11/15}$. However, this term determines the growth of the dimensions in question within a $\frak{p}$-congruence tower of arithmetic groups.
  \end{remark}
  
  \begin{remark}
The view on the isometry group of hyperbolic $5$-space via Hamilton quaternions and arithmetically defined subgroups is also taken up in \cite{Keller} or \cite{Parpau}.
 Within the framework of the theory of integral binary Hamiltonian forms, using the notion of the distance to a cusp, one finds in \cite{Parpau}  a treatment of the necessary reduction theory for $\Gamma_{\Lambda}$, much in the spirit of the work of Mendoza \cite{Me} in the case of the special linear group $SL_2$ over the ring of integers of an imaginary quadratic field. The volume of $X_G/\Gamma_{\Lambda}$ is determined by Emery in \cite{Emery} or \cite{ParEmery}.
  \end{remark}
  
  \begin{remark}\label{indef}In the case of an indefinite quaternion division $\Q$-algebra the situation is a bit more complicated, due to the fact that a boundary component of the $9$-dimensional manifold $X_G/\Gamma$  is a fibre bundle with base space a compact arithmetic $4$-dimensional quotient, say $Z_L/\Gamma_L(\frak{p}^e)$, with fibre a $4$-dimension torus. The spectral sequence in cohomology attached to this fibre bundle  degenerates at $E_2$ (\cite[Thm. 2.7]{S2}. Therefore the cohomology of a boundary component is given by the cohomology 
 $H^*(Z_L/\Gamma_L(\frak{p}^e), \tilde{H}^*(T^4, \C))$ of the base space with coefficients in the local system determined by the cohomology of the fibre. 
 Using Propositions \ref{cusps} and \ref{image} and taking into account the dimension of  $H^*(Z_L/\Gamma_L(\frak{p}^e), \tilde{H}^*(T^4, \C))$ one can derive assertions analogous to the 
 ones in Theorem \ref{final}.
 \end{remark}
 
  \begin{remark}\label{Eis} In both cases, within the interpretation of the cohomology  $H^q(X_G/\Gamma, \C)$ in terms of automorphic forms, by passing over to the projective limit over all congruence subgroups in the adelic framework, 
one can use the theory of Eisenstein series to describe the cohomology at infinity  (see, e.g., \cite{GrSch}, \cite{S2}). In particular, the possible existence of residues of Eisenstein  series explains the various dualities between degrees. We deal with this approach via automorphic forms in the  more general context of totally definite quaternion division algebras over a totally real algebraic number field in \cite{SchEis}.

From this point of view, the results in \cite{SchPac} show that there exist cohomology classes in $H^q(X_G/\Gamma, \C)$  which do not belong to the cohomology at infinity. Indeed, these classes are represented by cuspidal automorphic forms of different type.
\end{remark}

\subsection{Examples}\label{examples}
To have a family of examples at hand, we determine all quaternion algebras $Q$ over the field $\Q$ of rational numbers which ramify exactly at a given prime $p$ and
the unique archimedean place, to be denoted $\infty$. A non-archimedean place $v \in V_{\Q}$ corresponds to a unique  prime in $\Z$, and $\Q_v$ is the field $\Q_p$ of $p$-adic numbers. Using the  device of the Hilbert symbol and the related reciprocity law one derives the following (see, e.g., \cite[14.2]{Voight}):

\begin{proposition} Given a prime $p$ the quaternion algebras $Q(-1, -1\vert \Q)$ if $p =2$ resp. $Q(-1, -p \vert \Q)$ if $p \equiv 3 \mod 4$
resp. $Q(-q, -p \vert \Q)$ if $p \equiv 1 \mod 4$, where $q$ is a prime such that $q \equiv 3 \mod 4$ and $q$ is not a quadratic residue $\mod p$ ramify exactly at the places
 $\{\infty, p\}$. Each quaternion $\Q$-algebra $Q$ whose ramification set is  $\{\infty, p\}$ is isomorphic to one of the quaternion algebras as listed.
\end{proposition}

In these cases, the Eichler class number formula gives the exact value of the class number $h_D$ of $D$ (see \cite[Thm.30.1.5]{Voight}). Written out in terms of congruence conditions for the discriminant $p$ of $D$ (\cite[Exercise 30.6]{Voight}) it takes the form 
\begin{equation*}
h_D = \begin{cases}
1 & \text{if} \; p = 2, 3, \\
(p - 1)/12 &\text{if} \;p \equiv 1 \pmod {12}\\
(p + 7)/12 & \text{if}\; p \equiv  5 \pmod{12}\\
(p + 5)/12 & \text{if} \;p \equiv 7 \pmod{12}\\
(p + 13)/12&  \text{if} \;p \equiv  11 \pmod{12}.
\end{cases}
\end{equation*}
Examples of maximal orders in these definite quaternion $\Q$-algebras are given in \cite[15.5.7]{Voight} resp.  \cite{Ibu}.

\begin{remark} We make explicit in which cases the class number $h_D$ is equal to  one. Beyond the cases $p = 2, 3$, these are the following definite quaternion 
$\Q$-algebras
\begin{equation*}
\begin{cases}
p = 5 & D = Q(-q, -5 \vert \Q),  q \equiv 3 \pmod 4, q \;\text{non-residue} \pmod{5}\\
p = 7 & D = Q(-1, -7\vert \Q)\\
p = 13 & D = Q(-q, -13 \vert \Q),  q \equiv 3 \pmod 4, q \;\text{non-residue} \pmod{13}.
\end{cases}
\end{equation*}
Given any maximal order $\Lambda$ in the central simple algebra $M_2(D)$, $D$ as given in the list, the number of $\Gamma_{\Lambda}$-conjugacy classes of minimal parabolic $\Q$-subgroups of 
$G = SL_{\Lambda} \times_{\Z} \Q$ is equal to one. It follows that 
the corresponding $5$-dimensional orbifold $X_G/\Gamma_{\Lambda}$ has exactly one end.
This might be of interest in the context of determining one-cusped arithmetically defined hyperbolic orbifolds, see \cite[Sect. 5]{Stover}.

\end{remark}

\bibliographystyle{amsalpha}

\begin{thebibliography}{99}








\bibitem{Bo1} {Borel, A.}, {Introduction aux Groupes
Arithm\'etiques},
Hermann, Paris, 1969.







\bibitem{Emery}
{ Emery, V.} {\it Du volume des quotients arithmetiques de l`espace hyperbolique}, these no 1648, Universite de Fribourg (Suisse), 2009, available at http://www.unige.ch/math/folks/emery/ Emery.pdf.

\bibitem{Froe}
Fr\"ohlich, A. {\it Locally free modules over arithmetic orders}, J.  Reine  Angew. Math. \textbf{274-275} (1975), 112 - 124.


\bibitem{GrSch}
Grbac, N.,  Schwermer, J., {\it An exercise in automorphic cohomology - the case $GL_2$ over a quaternion algebra}, In: Arithmetic Geometry and Automorphic Forms, Volume in honor of the 60th birthday of Stephen S. Kudla. Advanced Lectures in Mathematics, vol. 19,  International Press and the Higher Education Press of China, Beijing, 2011, pp. 209--252. 




\bibitem{Ha4} {Harder, G.}, {\it A Gauss--Bonnet formula for discrete
arithmetically defined groups,} Ann. Sci. \'Ec. Norm. Sup.(4) \textbf{4} (1971),
409 - 455.

\bibitem{Hel} {Helgason, S.}, {Differential Geometry, Lie Groups, and Symmetric
Spaces},
 New York and London: Academic Press, 1962.
 
 \bibitem{Ibu}
 {Ibukiyama,T. }, {\it On maximal orders of division quaternion algebras
over the rational number field with certain optimal embeddings}, Nagoya
Mathematical Journal, \textbf{88} (1982), 181-195. 
 
\bibitem{JS} 
{Jantzen, J.C., Schwermer, J.}, 
{Algebra}, 2nd edition, Berlin, Heidelberg: Springer, 2014.


\bibitem{Keller}
{Kellerhals, R.}, {\it Quaternions and some global properties of hyperbolic $5$-manifolds}, Canad. J. Math., \textbf{55} (2003), 1080-1099.



\bibitem{Kionkethesis}
Kionke, S., {\it Lefschetz numbers of involutions on arithmetic subgroups of inner forms of the special linear group}, Thesis, University of Vienna, Vienna, 2012.



\bibitem{Kneser}
{Kneser, M.}, {\it Starke Approximation in algebraischen Gruppen I},
J.  Reine  Angew. Math. \textbf{218} (1965), 190 - 203.

\bibitem{Knus}
Knus, M.A. {Quadratic and Hermitian forms over rings}, Grundlehren math. Wiss., vol. 291,  Berlin, Heidelberg: Springer, 1991.



\bibitem{Koch}{Koch, S.}, {\it On the special linear group over orders in quaternion division algebras}, 
J. Number Theory {\textbf 181} (2017), 147 - 163.


\bibitem{KSch}
{Koch, S. and Schwermer, J.},
{\it On arithmetic quotients of the group $SL_2$ over a quaternion division $k$-algebra}, to appear,  Forum Mathematicum, 2024. https://doi.org/10.1515/forum-2023-0422


\bibitem{KOS}
 {Krafft, V. and Osenberg, D.}, {\it Eisensteinreihen f\"ur einige arithmetisch definierte
Untergruppen von $SL_2(\mathbb{H})$}, Math. Z. \textbf{204} (1990), 425 - 446.



\bibitem{Lacher} {Lacher, C.}, {\it Zur Geometrie der affin-algebraischen reduzierten Norm-$1$ Gruppe einer zentral einfachen Algebra endlicher Dimension mit einem algebraischen Zahlk\"orper als Zentrum}, thesis,  University of Vienna, 2005.



\bibitem{MacReid}
{Maclachlan, C. and Reid, A.}, {The arithmetic of hyperbolic $3$-manifolds}, Graduate Texts  Maths, vol. 219, New York: Springer, 2003.

\bibitem{Me} {Mendoza, E.}, {\it Cohomology of $PGL_2$ over imaginary quadratic integers}, Bonner Math. Schriften 128. Math. Institut, Universit\"at Bonn, Bonn 1980.


\bibitem{ParEmery} 
{Parkkonen, J. and Paulin, F.} {\it On the arithmetic and geometry of binary Hamiltonian forms (with an appendix by V. Emery)}, Algebra and Number Theory, {\bf 7} (2013), 75-115.

\bibitem{Parpau}{ Parkkonen, J. and Paulin, F.}, {\it Integral Binary Hamiltonian  forms and their waterworlds},
Conform. Geom. Dyn. \textbf{25} (2021), 126 - 169.




\bibitem{Re} {Reiner, I.}, {Maximal Orders}. London--New York: Academic Press
1975.





\bibitem{S2} {Schwermer, J.}, {\it Kohomologie arithmetisch
definierter Gruppen und Eisensteinreihen}, Lecture Notes in Math., 988,
Berlin--Heidelberg--New York: Springer-Verlag 1983.

 \bibitem{SchPac}{Schwermer, J.},
 {\it Non-vanishing  square-integrable automorphic cohomology classes -   the case $GL(2)$ over a central division algebra}, Pacific J. Mathematics, \textbf{306} (2020) 321 - 355.
 
 \bibitem{Schbook} {Schwermer, J.,} {Reduction Theory and Arithmetic Groups}, New Math. Monographs vol. 45, Cambridge: Cambridge U Press, 2023.
 

 
 \bibitem{SchEis} {Schwermer, J.,} {\it Eisenstein cohomology - the case $SL_2$ over a totally definite  quaternion  $k$-algebra}, in preparation.
 


\bibitem{Span}
Spanier, E.H., Algebraic Topology, McGraw-Hill, New York, 1966.

\bibitem{Stover}
Stover, M., {\it On the number of ends of rank one locally symmetric spaces}, Geometry  Topology. \textbf{17} (2013), 905-924.

\bibitem{Tama}
Tamagawa, ~T., On quadratic forms and pfaffians, J. Fac, Sci. Tokyo, Sect. I, \textbf{24} (1977), 213-219.


\bibitem{T1} {Tits, J.}, {\it Classification of algebraic semisimple groups}, In: Algebraic Groups and Discontinuous Subgroups, Proc. Symp. Pure Math. 9 (1966), pp. 33-62.

 
 \bibitem{Voight} 
 {Voight, J.} {Quaternion Algebras}, Graduate Texts Math. 288, Cham: Springer, 2021.
 
 
 \end{thebibliography}

  \section{conflicts of interest}
 The author has no competing interests to declare that are relevant to the content of this article.
 
 \end{document}